\documentclass[letterpaper]{article}  
\usepackage[left=1in, right=1in, bottom=.8in, top=0.8in, headsep=0in, footskip=.2in]{geometry}

\usepackage{amsmath}
\usepackage{amsthm}
\usepackage{mathtools}
\usepackage{amssymb}

\usepackage{algorithm}
\usepackage{algorithmic}
\usepackage{caption}
\usepackage{subcaption}
\usepackage{xcolor}

\newtheorem{theorem}{Theorem}
\newtheorem{proposition}{Proposition}
\newtheorem{definition}{Definition}

\newtheorem{remark}{Remark}
\newcommand{\R}{\mathbb{R}}
\newcommand{\x}{\textbf{x}}

\title{\LARGE \bf
Harmonic Control Lyapunov Barrier Functions for Constrained Optimal Control with Reach-Avoid Specifications
}

\author{Amartya Mukherjee, Ruikun Zhou, Haocheng Chang, and Jun Liu
\thanks{$^{1}$Amartya Mukherjee, Ruikun Zhou, Haocheng Chang, and Jun Liu are with the Department of Applied Mathematics, University of Waterloo, Waterloo, Ontario, Canada N2L 3G1 (email: {\tt\small (a29mukhe, ruikun.zhou, h48chang, j.liu)@uwaterloo.ca}).}%
}
\date{}

\begin{document}

\maketitle
\thispagestyle{empty}
\pagestyle{empty}

\begin{abstract}
This paper introduces harmonic control Lyapunov barrier functions (harmonic CLBF) that aid in constrained control problems such as reach-avoid problems. Harmonic CLBFs exploit the maximum principle that harmonic functions satisfy to encode the properties of control Lyapunov barrier functions (CLBFs). As a result, they can be initiated at the start of an experiment rather than trained based on sample trajectories. The control inputs are selected to maximize the inner product of the system dynamics with the steepest descent direction of the harmonic CLBF. Numerical results are presented with four different systems under different reach-avoid environments. Harmonic CLBFs show a significantly low risk of entering unsafe regions and a high probability of entering the goal region. 
\end{abstract}

\section{Introduction}

For safety-critical systems, ensuring stability and safety is of paramount importance. One approach to address this issue for such systems is to model it as a constrained optimal control problem.
The objective of constrained optimal control is to design a control system that optimizes a notion of performance while satisfying a set of constraints. These constraints are usually designed to prevent catastrophic events, ensure safety, and take into multiple objectives with bounds. Furthermore, with experiments conducted in real life, resetting the environment if the agent goes outside a domain can be expensive \cite{nguyen2023reset}.
One of the most common scenarios for safety-critical systems is the reach-avoid problem, which can be regarded as a subset of constrained optimal control problems where a trajectory aims to reach a goal state while avoiding a set of unsafe states. 

In recent years, there has been a growing interest in reach-avoid problems among the control theory and reinforcement learning community. 
Such methods involve learning control Lyapunov functions to certify reachability and control barrier functions that satisfy avoidability \cite{gurriet2018towards,choi2020reinforcement}.
In reinforcement learning, agents learn optimal policies through exploration of the environment with unknown dynamics \cite{sutton2018reinforcement}. Constraints are imposed using penalty methods \cite{heess2017emergence} and Lagrangian methods \cite{ray2019benchmarking}. In these methods, two value functions are computed to estimate cumulative rewards and cumulative costs. 

The problem with using two certificates or two value functions is that it leads to conflicts in the control policy whether to satisfy reachability or avoidability at a given point in time.
To unify the conflicts, reachability and safety have been combined into a single certificate called the control Lyapunov barrier function (CLBF). It is shown in
\cite{MENG2022110478} that theoretically there exists a single Lyapunov barrier function for such reach-avoid type specifications with stability and safety guarantees, provided that such specifications are robustly satisfiable.  
Moreover, \cite{dawson2021safe} introduces robust CLBFs and derives a QP-based controller, in which the CLBF is computed using neural networks by encoding the constraints in their loss function.
In a similar manner but with reinforcement learning, the Lyapunov barrier-based actor-critic method is proposed in~\cite{du2023reinforcement} to find a controller that satisfies both certificates with a CLBF that is obtained by solving the Lagrangian relaxation of the constrained optimization problem.
Unfortunately, both methods do not provide guarantees that their value function $V$ satisfies all the CLBF constraints, and this is evidenced by their experiments. For example, in 2D quadrotor control problems, the contour plots of the derived CLBF do not completely cover the unsafe regions.
Our work addresses this issue by enforcing the CLBF constraints as boundary conditions for harmonic functions. Numerical results demonstrate that this technique indeed improves safety rates, compared with the results reported in \cite{dawson2021safe}. 

Harmonic functions are solutions to the Laplace equation, which is a second-order linear elliptic partial differential equation (PDE). They satisfy the maximum principle on a compact set, thus making it easier to impose CLBF constraints. They do not have any critical points other than saddle points in the interior of their domain, which makes it easy to derive optimal control strategies. 
Harmonic functions have been used in the control literature to derive potential fields in past works. In the work of 
\cite{kim1992real}, harmonic potential functions with panel methods were used to build potential fields over a task space. The work of \cite{akishita1990lapace} used harmonic functions with complex variables and conformal mappings to achieve moving obstacle avoidance. Furthermore, \cite{connolly1990path} combines harmonic functions using superposition to derive potential fields.

In this paper, we will focus on reach-avoid problems, where an agent must reach a goal region while avoiding unsafe regions. For example, the unsafe regions could be regions on a floor that have holes or pillars. We achieve this by introducing harmonic control Lyapunov barrier functions (harmonic CLBFs) that use Laplace's equation to encode the properties of control Lyapunov barrier functions (CLBFs).
Furthermore, Laplace's equation can be solved using numerical methods such as finite element methods (FEMs) and do not be trained using neural networks based on sample trajectories as done in prior work, thus significantly reducing the computational cost of reach-avoid experiments. 

In all, the main contributions of this paper are threefold. Firstly, we propose a novel method to unify harmonic functions with CLBFs to provide both stability and safety guarantees. Secondly, under this framework, we show that optimal controllers for reach-avoid tasks can be derived using gradient-based methods directly with no need for searching such controllers by training a neural network or solving optimal problems. Lastly, we conduct a set of numerical experiments that show a significantly low risk of trajectories entering unsafe regions with the proposed approach.

\section{Preliminaries}

Throughout this work, we consider a nonlinear control system of the form 
\begin{equation}
\dot{x} = f(x,u), \quad x(0)=x_{0},
\label{eq:dyn}
\end{equation}
where $x \in \mathcal{S}\subset \mathbb{R}^n$ is the state of the system, and $u \in U \subset \mathbb{R}^{m}$ is the control input. 

Let $\mathcal{S}$ be a compact subspace of $\R^n$ denoting the space of all admissible states in a control problem. Let $\mathcal{S}_{goal}\subset\mathcal{S}$ be an open set denoting the space of states that the controller aims to reach, and let $\mathcal{S}_{unsafe}$ be a compact subset of $\mathcal{S}$ denoting the space of states the controller aims to avoid. Define $\mathcal{S}_{safe}=\mathcal{S}\backslash\{\overline{\mathcal{S}_{goal}}\cup \mathcal{S}_{unsafe}\}$. 

\subsection{Control Lyapunov Barrier Functions}

\begin{definition}[Control Lyapunov Barrier Function \cite{dawson2021safe,du2023reinforcement}]
    A function $V:\mathcal{S}\to\R$ is a control Lyapunov barrier function (CLBF) if, for some constants $c,\lambda>0$:
    \begin{enumerate}
        \item $V(s)=0$ for all $s\in \mathcal{S}_{goal}$; 
        \item $V(s)>0$ for all $s\in \mathcal{S}\backslash \mathcal{S}_{goal}$; 
        \item $V(s)\geq c$ for all $s\in \mathcal{S}_{unsafe}$; 
        \item $V(s)<c$ for all $s\in \mathcal{S}\backslash\mathcal{S}_{unsafe}$; 
        \item There exists a controller $\pi:\,S_{safe} \rightarrow U$ such that $\langle \nabla V(s),f(s,\pi(s))\rangle+\lambda V(s)\leq 0$ for all $s\in S_{safe}$.
    \end{enumerate}
\label{def:clbf}
\end{definition}

\subsection{Harmonic Functions}

\begin{definition}[Harmonic Function \cite{gilbarg1983elliptic}]
    A harmonic function $u$ is a $C^2$-function that satisfies the Laplace equation $\nabla^2u=\nabla\cdot\nabla u=0$
\end{definition}

In the following theorems, we introduce properties of harmonic functions that are of importance in the field of elliptic PDEs, differential geometry, and complex analysis.

\begin{theorem}[Mean Value Theorem \cite{protter1983maximum}]
\label{thm:mean_value}
    If $u$ is harmonic on a domain $\mathcal{S}$, then $u$ satisfies the mean value property in $\mathcal{S}$. Furthermore, if $x\in\mathcal{S}$ and $r>0$ are such that $\overline{B_r(x)}\subset\mathcal{S}$, where $B_r(x)$ is a sphere of radius $r$ centered at $x$, then
    \begin{equation}
        u(x)=\frac{\int_{B_r(x)}u(y)dy}{\int_{B_r(x)}dy}.
    \end{equation}
\end{theorem}

This theorem is essential for proving Theorems \ref{thm:strong_max} and \ref{thm:weak_max} that provide bounds on harmonic functions.

\begin{theorem}[Strong Maximum Principle) \cite{gilbarg1983elliptic}]
\label{thm:strong_max}
    If $u$ is harmonic on a domain $\mathcal{S}$ and $u$ has its maximum in $\mathcal{S}\backslash\partial\mathcal{S}$, then $u$ is constant.
\end{theorem}

\begin{theorem}[Weak Maximum Principle \cite{gilbarg1983elliptic}]
\label{thm:weak_max}
    If $u\in C^2(\mathcal{S})\cup C^1(\partial\mathcal{S})$ is harmonic in a bounded domain $\mathcal{S}$, then 
    \begin{equation}
    \max_{\overline{\mathcal{S}}}u=\max_{\partial\mathcal{S}}u.
    \end{equation}
\end{theorem}

This shows that Theorems \ref{thm:strong_max} and \ref{thm:weak_max} can be exploited to derive a CLBF that has its maximum in unsafe regions and minimum in goal regions.

\section{Harmonic CLBF}

In this work, we explore the intersection between CLBF and harmonic functions. By exploiting the maximum principle, we can derive a function $V:S\to\R$ that satisfies properties (1)--(4) for CLBF by imposing boundary conditions on the CLBF.

\begin{definition}[Harmonic CLBF]
    \label{def:harmonicclbf}
    A function $V\in C^2(\mathcal{S})\cup C^1(\partial \mathcal{S})$ is a harmonic CLBF if it satisfies the following conditions:
    \begin{enumerate}
        \item $\nabla^2V(s)=0$ for all $s\in \mathcal{S}_{safe}$
        \item $V(s)=0$ for all $s\in\overline{\mathcal{S}_{goal}}$
        \item $V(s)=c$ for all $s\in\partial\mathcal{S}\cup \mathcal{S}_{unsafe}$.
    \end{enumerate}
    In other words, $V(s)$ is a solution to the boundary value problem given above.
\end{definition}

\begin{theorem}
    All harmonic CLBFs $V$ satisfy properties (1)--(4) in Definition \ref{def:clbf}.
\end{theorem}
\begin{proof}
    Properties (1) and (3) in Definition \ref{def:clbf} are trivially satisfied as they are set as boundary conditions from properties (2) and (3) in Definition \ref{def:harmonicclbf}.
    
    If $V(s)\leq 0$ for some $s\in\mathcal{S}\backslash\mathcal{S}_{goal}$ ($V(s)\geq c$ for some $s\in\mathcal{S}\backslash\mathcal{S}_{unsafe}$), then $V(s)$ has its minimum (maximum) in the interior of $\mathcal{S}_{safe}$. By Theorem \ref{thm:strong_max}, $V$ is then constant in $\mathcal{S}_{safe}$, which contradicts the non-constant boundary conditions imposed on it. Thus, harmonic CLBFs satisfy properties (2) and (4) in Definition \ref{def:clbf}.
\end{proof}
\begin{proposition}
    Every reach-avoid problem admits a unique harmonic CLBF.
\end{proposition}

\begin{proof}
    Consider an arbitrary reach-avoid problem characterized by $S,S_{goal},S_{unsafe}$.
    Assume for a contradiction that a reach avoid problem has two distinct harmonic CLBFs, $V_1(s)$ and $V_2(s)$.
    
    Define the function $U(s)=V_1(s)-V_2(s)$.
    The Laplacian of $U(s)$ is $\nabla^2U=\nabla^2V_1-\nabla^2V_2=0$ on $S_{safe}$.
    As a result, $U$ is a function satisfying $\nabla^2 U=0$ on $S_{safe}$, $U=0$ on $\overline{S_{goal}}$, and $U=0$ on $\partial\mathcal{S}\cup \mathcal{S}_{unsafe}$.
    By Theorem \ref{thm:weak_max}, $U$ has its maximum and minimum of $0$ on the boundary, so $U=0$ on $S_{safe}$. Thus, the harmonic CLBF is unique.
\end{proof}

We now introduce a sufficient condition for verifying property (5) of the CLBF conditions in Definition~\ref{def:clbf}.

\begin{theorem}
    \label{thm:lambda}
    Let $V$ be a harmonic CLBF, and consider the deterministic nonlinear system~\eqref{eq:dyn}. If there exists a positive constant $\lambda$ such that 
    \begin{equation}
        \label{eq:lambda}
        \lambda \leq -\sup_{x\in\mathcal{S}_{safe}}\inf_{u\in U}\frac{\langle f(x,u),\nabla V(x)\rangle}{V(x)},
    \end{equation}
    then property (5) in Definition~\ref{def:clbf} holds.
\end{theorem}

\begin{proof}
Define an optimal controller $\pi^*(x)$ as 
$$
\pi^*(x) = \arg\inf_{u\in U}\langle f(x,u),\nabla V(x)\rangle. 
$$
Since $\lambda V(x)\geq 0$ for all $x\in\mathcal{S}_{safe}$, 
it follows that 
\begin{align*}
\frac{\langle f(x,\pi^*(x)),\nabla V(x)\rangle}{V(x)} &\le \inf_{u\in U}\frac{\langle f(x,u),\nabla V(x)\rangle}{V(x)}    \\
& \le \sup_{x\in\mathcal{S}_{safe}}\inf_{u\in U}\frac{\langle f(x,u),\nabla V(x)\rangle}{V(x)} \\
&\le -\lambda,
\end{align*}
where the last inequality is precisely (\ref{eq:lambda}). We have verified property (5) in Definition~\ref{def:clbf}.

\end{proof}


Furthermore, a sufficient condition for a system to avoid $\mathcal{S}_{unsafe}$ is:
\begin{equation}
    \label{eq:sup_inf}
    \sup_{x\in\mathcal{S}_{safe}}\inf_{u\in U}\langle f(x,u),\nabla V(x)\rangle\leq 0 
\end{equation}
Since $\mathcal{S}_{unsafe}$ is a closed set, this means $V(x)<c$ for all $x\in\mathcal{S}_{safe}$. Furthermore, since $\frac{\partial}{\partial t}V(x(t))\leq 0$, this means for all $t\in[0,\infty)$, $x(t)$ will never approach a point where $V(x)=c$, meaning, it will never approach a point in $\mathcal{S}_{unsafe}$.

\subsection{Avoiding the saddle points}

As shown in \cite{yanushauskas1969zeros}, all critical points in $V(x)$ are saddle points. Random noise is added in~\cite{ge2015escaping} to parameters while performing gradient descent to guarantee that the algorithm converges to a local minimum. The authors of \cite{lee2016gradient} show that gradient descent with random parameter initialization asymptotically avoids saddle points. So in this paper, we compare deterministic control inputs with stochastic control inputs to see which method shows better convergence to $\mathcal{S}_{goal}$.

At every point $x\in \mathcal{S}_{safe}$, the control input is selected as $u=\text{argmin}_{u\in U}\langle f(x,u),\nabla V(x)\rangle+z$ where $z$ is a small noise sampled from a normal distribution $N(0,\sigma^2)$ to mitigate local optima.

\subsection{Conditions on the noise}
In order to preserve the properties of the CLBF, we introduce bounds that must be imposed on $z$ with Taylor series expansion so that the trajectory avoids unsafe regions when the first-order terms dominate.


    Let $\Delta t$ be the time step size used for numerical simulation. Consider the first-order expansion of $x(t+\Delta t)$ as
    \begin{equation}
        x(t+\Delta t)=x+f(x,u)\Delta t+O(\Delta t^2).
    \end{equation}
    By adding some noise $z$ to the control input to avoid saddle points, we have
    \begin{equation}
        x(t+\Delta t)=x+f(x,u+z)\Delta t+O(\Delta t^2).
    \end{equation}
    Then we apply again the first-order expansion of $f(x,u+z)$, yields
    \begin{equation}
        x(t+\Delta t)=x+f(x,u)\Delta t+f_u(x,u)^Tz\Delta t+O(\Delta t^2).
    \end{equation}
    Also, we consider the first-order expansion of $V(x(t+\Delta t))$:
    \begin{align}
        V(x(t+\Delta t))=&V(x)+\langle\nabla V,f(x,u)+\nabla_uf(x,u)^Tz\rangle\Delta t\nonumber\\&+O(\Delta t^2).
    \end{align}
    To have the safety guarantees, the condition $V(x(t+\Delta t))<c$ needs to be satisfied, that is,
    \begin{equation}
        \langle\nabla V(x),f(x,u)+\nabla_uf(x,u)^Tz\rangle<\frac{c-V(x)}{\Delta t}.
    \end{equation}
    Expanding and simplifying this expression gives the following bounds:
    \begin{equation}
        \label{eq:noise_bounds}
        \langle\nabla_uf(x,u)\nabla V(x),z\rangle<\frac{c-V(x)}{\Delta t}-\langle\nabla V(x), f(x,u)\rangle.
    \end{equation}
Using the Cauchy-Schwarz inequality, a sufficient way to bound the noise $z$ is
\begin{equation}
    ||\nabla_uf(x,u)\nabla V(x)||_2||z||_2<\frac{c-V(x)}{\Delta t}-\langle\nabla V(x), f(x,u)\rangle.
\end{equation}

If $f$ is a control-affine function $f(x,u)=f_1(x)+f_2(x)u$, then we can bound $z$ as:
\begin{equation}
    \label{eq:z_upper_bound}
    ||z||_2<\frac{\frac{c-V(x)}{\Delta t}-\langle\nabla V(x), f(x,u)\rangle}{||f_2(x)||_2||\nabla V(x)||_2}.
\end{equation}
Since Theorem \ref{thm:lambda} requires that under the optimal controller $\pi^*(x)$, $\langle\nabla V(x), f(x,\pi^*(x))\rangle\leq 0$, this shows that the bound imposed on $z$ is positive. 

\begin{remark}
    Numerical solutions in section \ref{sec:results} show that in harmonic CLBFs, the distinction between safe and unsafe regions is unclear. To make the distinction more transparent, we replace the CLBF as a solution to Laplace's equation with a solution to Poisson's equation $\nabla^2V=-6$ in this paper, meaning $V$ is a superharmonic function. However, this method also poses a risk that $V$ has local minima in the interior of its domain \cite{axler2013harmonic}. This is undesirable as local minima are harder to escape compared to saddle points. We will compare harmonic CLBFs with superharmonic CLBFs with numerical results.
\end{remark}

\section{Numerial Results}
\label{sec:results}

In this section, we will explore two different reach-avoid environments, one with four small unsafe regions and the other with two big unsafe regions, for three different systems: Roomba, DiffDrive, and CarRobot \cite{ModernRobotics}. Except that, we also perform a reach-avoid task for 2D Quadrotor \cite{dawson2021safe}. The dynamics of each of these systems are provided in Appendix I along with control inputs that minimize the expression in Equation \eqref{eq:sup_inf} for any $x\in \mathcal{S}_{safe}$. We use $c=1$ in the CLBF defined in Definition~\ref{def:clbf} for all numerical results.

For each environment and system, we compute the harmonic ($\nabla^2V=0$) and superharmonic ($\nabla^2V=-6$) CLBF numerically using finite element methods. For both of the CLBFs, we compute the trajectories of the system with 1,000 different randomly initialized initial conditions and count the number of time steps (of size $\Delta t=0.1$) it takes for the system to reach $\mathcal{S}_{goal}$. We test this setting with both deterministic ($\sigma=0$) and stochastic ($\sigma=0.1$, clipped by the upper bound in Equation \eqref{eq:z_upper_bound}) controllers and report the mean time taken to reach the goal ($\mu_T$), the standard deviation of the time taken to reach the goal ($\sigma_T$), the number of times the system ends in an unsafe region (not included in $\mu_T,\sigma_T$), and the number of times the system does not reach $\mathcal{S}_{goal}$ after 1,000 time steps (not included in $\mu_T,\sigma_T$).

\subsection{Problem I}

In this problem, we explore an environment that contains a goal region near the origin and four unsafe regions in the interior of the domain. The results with three car-like dynamical systems, Roomba, DiffDrive, and CarRobot, are reported.
\begin{align}
    S&=[-1,1]\times[-1,1]\\
    \mathcal{S}_{goal}&=[-0.1,0.1]\times[-0.1,0.1]\\
    \mathcal{S}_{unsafe}&=\partial S\cup C_1\cup C_2\cup C_3\cup C_4,
\end{align}
with the subdomains of the unsafe region given by
\begin{align}
    C_1&=[-0.5,-0.3]\times[-0.5,-0.3]\\
    C_2&=[-0.5,-0.3]\times[0.3,0.5]\\
    C_3&=[0.3,0.5]\times[-0.5,-0.3]\\
    C_4&=[0.3,0.5]\times[0.3,0.5].
\end{align}
Initial conditions are sampled as:
\begin{align}
    x(0),y(0)&\sim U[-0.9,-0.6]\cup[0.6,0.9]\\
    \theta(0)&\sim U[0,2\pi].
\end{align}
We derive a harmonic CLBF using finite element methods using DOLFIN \cite{dolfin}. We used piecewise linear trial functions and a triangular mesh. The domain has been divided into 5,000 triangular meshes of equal area. This numerical solution is plotted in 2D view in Fig. \ref{fig:1_0_2d}.

\begin{figure*}[ht]
    \centering
     \begin{subfigure}[b]{0.24\textwidth}
         \centering
         \includegraphics[width=\linewidth,trim={0 0 0 3cm},clip]{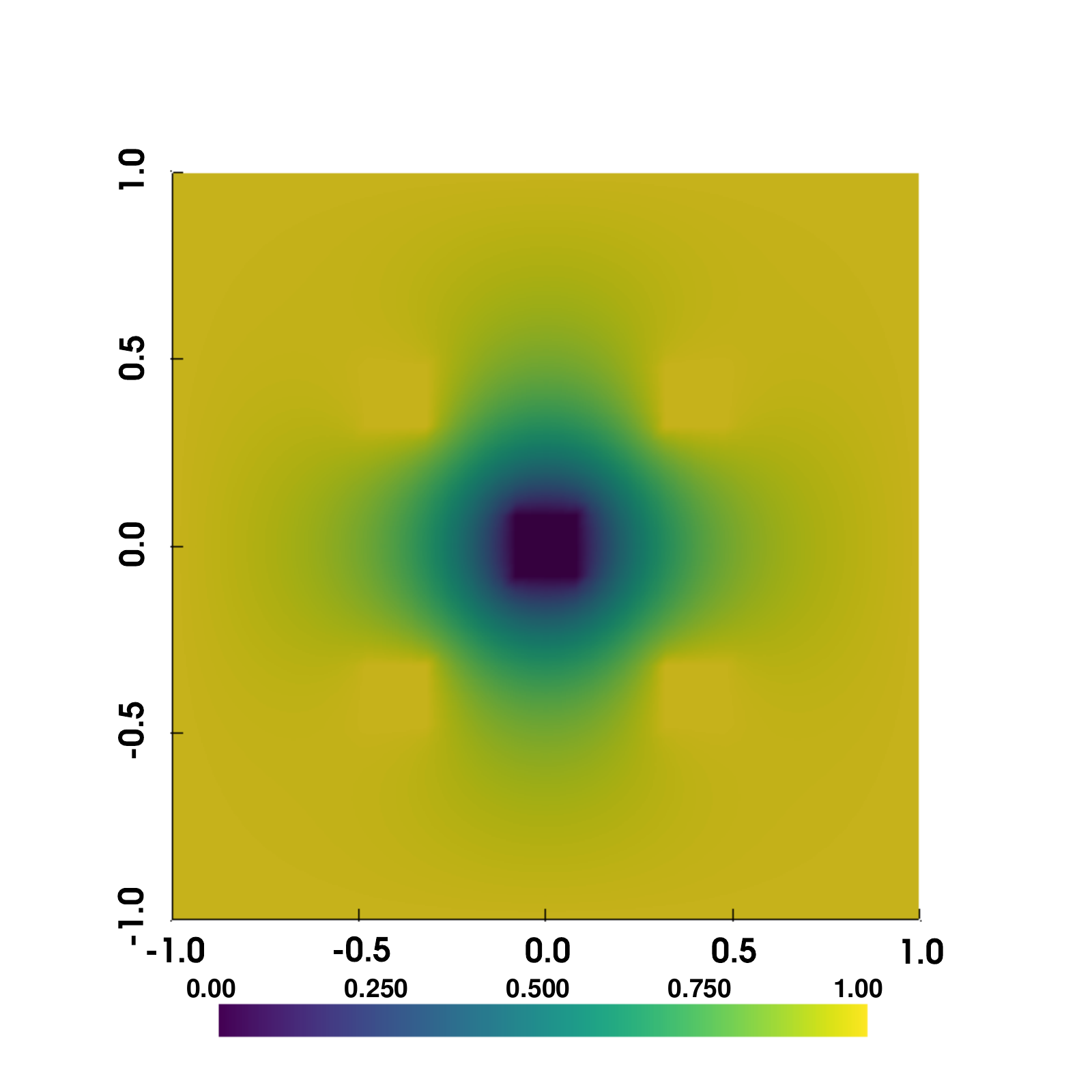}
        \caption{Problem I with $\nabla^2V=0$}
        \label{fig:1_0_2d}
     \end{subfigure}
     \hfill
     \begin{subfigure}[b]{0.24\textwidth}
         \centering
         \includegraphics[width=\linewidth,trim={0 0 0 3cm},clip]{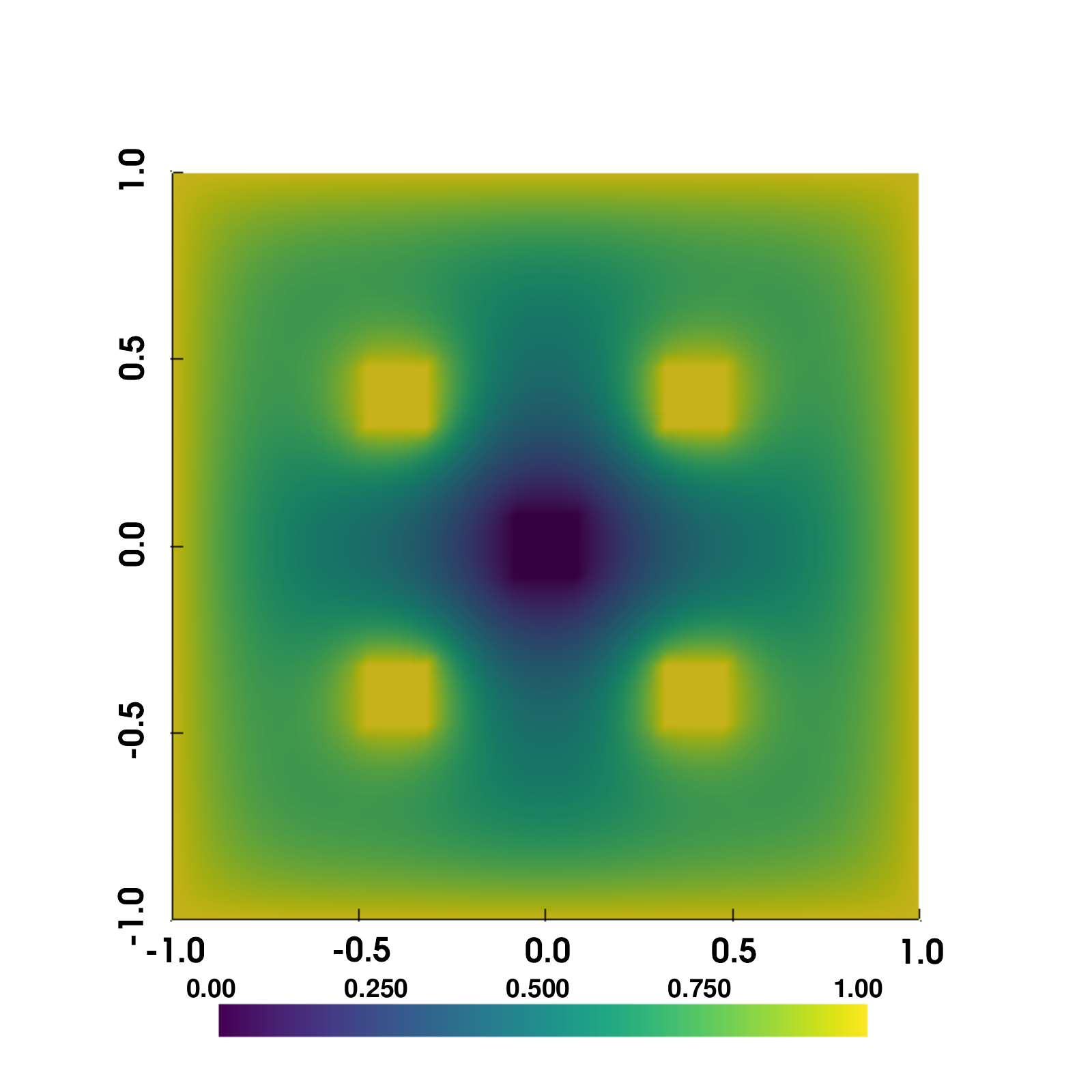}
        \caption{Problem I with $\nabla^2V=-6$}
        \label{fig:1_6_2d}
     \end{subfigure}
     \hfill
     \begin{subfigure}[b]{0.24\textwidth}
         \centering
         \includegraphics[width=\linewidth,trim={0 0 0 3cm},clip]{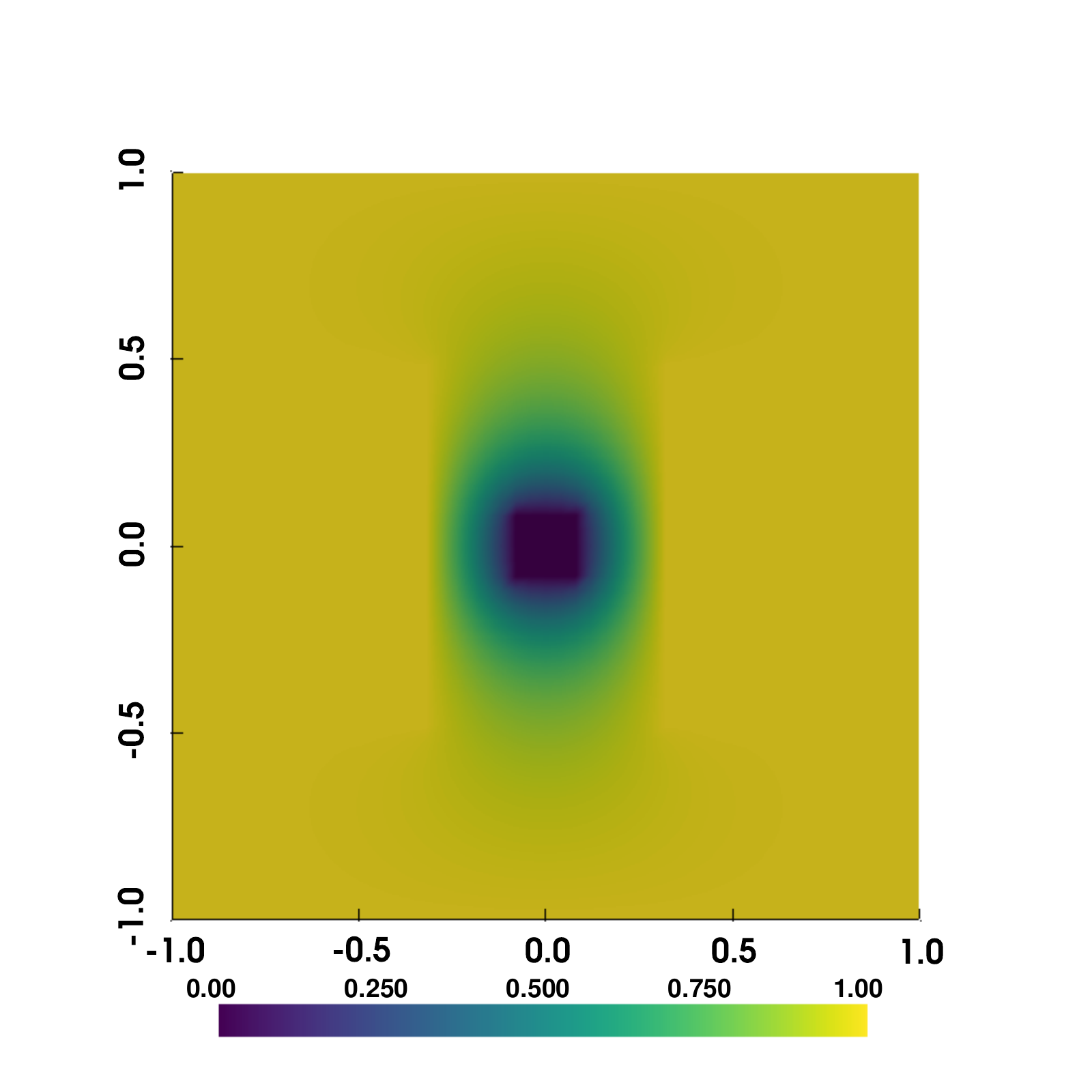}
        \caption{Problem II with $\nabla^2V=0$}
        \label{fig:2_0_2d}
     \end{subfigure}
     \hfill
     \begin{subfigure}[b]{0.24\textwidth}
         \centering
        \includegraphics[width=\linewidth,trim={0 0 0 3cm},clip]{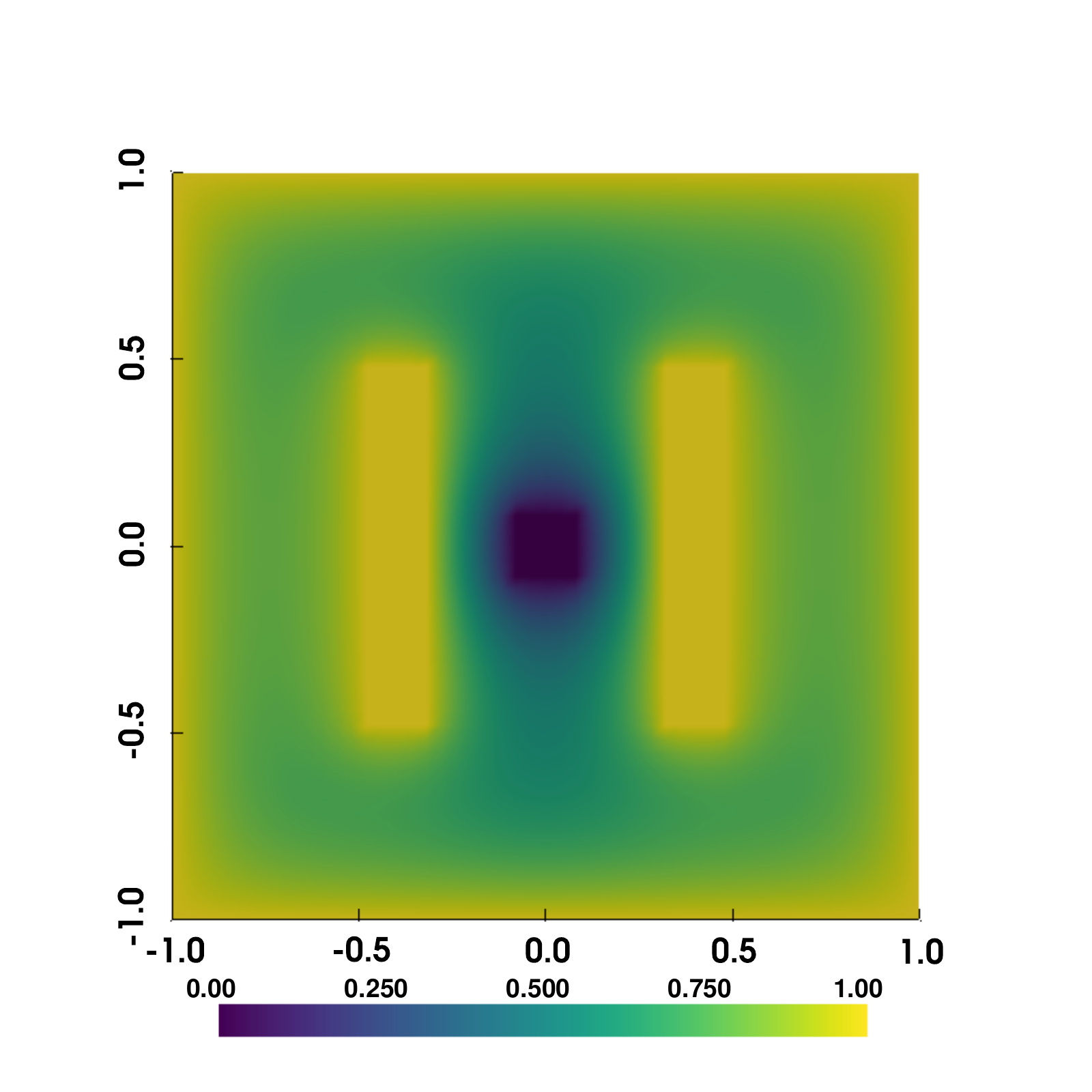}
        \caption{Problem II with $\nabla^2V=-6$}
        \label{fig:2_6_2d}
     \end{subfigure}
        \caption{Contour plots of the CLBFs for problems I and II. (a) \& (b) The plots for problem I with the CLBF obtained by solving Laplace equation (left) and Poisson's equation (right) respectively; (c) \& (d) The plots for problem II with the CLBF obtained by solving Laplace equation (left) and Poisson's equation (right) respectively.}
\end{figure*}

Although in theory, the solution should take values less than 1.0 outside the unsafe region, the distinction is highly unclear in the numerical solution. To clarify the distinction, we replace the Laplace equation with the Poisson equation $\nabla^2V=-6$. We plotted the solution in 2D view in Fig. \ref{fig:1_6_2d}.
This solution shows a better distinction between safe and unsafe regions, but it shows more points where $V(x)$ has a local optimum. Local optima are undesirable as $\langle f(x,u),\nabla V(x)\rangle=0$ in those points so $V(x)$ is unlikely to show a significant decrease.
\begin{table}[ht]
    \centering
    \begin{tabular}{c|c|c||c|c|c|c}
        System & $\nabla^2V$ & $\sigma$ & $\mu_T$ & $\sigma_T$ & unsafe & no reach\\
        \hline
        Roomba & 0 & 0 & 97.278 & 65.607 & 0 & 105 \\
        Roomba & -6 & 0 & 58.224 & 50.388 & 0 & 459\\
        \textbf{Roomba} & \textbf{0} & \textbf{0.1} & \textbf{114.706} & \textbf{81.527} & \textbf{0} & \textbf{0} \\
        Roomba & -6 & 0.1 & 139.606 & 180.247 & 0 & 185\\
        DiffDrive & 0 & 0 & {301.403} & {216.495} & {0} & {95}\\
        DiffDrive & -6 & 0 & 127.045 & 115.348 & 0 & 492\\
        \textbf{DiffDrive} & \textbf{0} & \textbf{0.1} & \textbf{239.759} & \textbf{220.689} & \textbf{0} & \textbf{2}\\
        DiffDrive & -6 & 0.1 & 158.837 & 167.421 & 0 & 5\\
        CarRobot & 0 & 0 & 452.125 & 323.942 & 0 & 193 \\
        \textbf{CarRobot} & \textbf{-6} & \textbf{0} & \textbf{277.398} & \textbf{253.148} & \textbf{0} & \textbf{47} \\
        CarRobot & 0 & 0.1 & 21.200 & 69.889 & 0 & 930 \\
        CarRobot & -6 & 0.1 & 26.887 & 88.174 & 0 & 920 \\
    \end{tabular}
    \caption{Results for Problem I. Each system was run with 1,000 different randomly initialized initial conditions under harmonic and superharmonic CLBFs with deterministic and stochastic controllers.}
    \label{tab:CLBF1}
\end{table}

The numerical results are listed in Table \ref{tab:CLBF1}. Roomba achieves the best results under a harmonic CLBF with a stochastic control policy as all the trajectories converged to $\mathcal{S}_{goal}$. DiffDrive achieves the best results under a harmonic CLBF with a stochastic control policy with a $99.8\%$ safety rate and similar results with a superharmonic CLBF with a stochastic control policy with a $99.5\%$ safety rate. CarRobot achieves the best results under a superharmonic CLBF with a deterministic policy with a $95.3\%$ safety rate.
\subsection{Problem II}
In this problem, we explore an environment that contains a goal region near the origin and two unsafe regions in the interior of the domain. In a similar vein to problem I, the three models are tested with this reach-avoid task. 
\begin{align}
    S&=[-1,1]\times[-1,1]\\
    \mathcal{S}_{goal}&=[-0.1,0.1]\times[-0.1,0.1]\\
    \mathcal{S}_{unsafe}&=\partial S\cup C_5\cup C_6,
\end{align}
with the subdomains of the unsafe region given by
\begin{align}
    C_5&=[-0.5,-0.3]\times[-0.5,0.5]\\
    C_6&=[0.3,0.5]\times[-0.5,0.5].
\end{align}
Initial conditions are sampled as:
\begin{align}
    x(0)&\sim U[-0.9,-0.6]\cup[0.6,0.9]\\
    y(0)&\sim U[-0.3,0.3]\\
    \theta(0)&\sim U[0,2\pi].
\end{align}
The numerical solution to the harmonic CLBF for this problem is plotted in 2D view in Fig. \ref{fig:2_0_2d}.


To make the distinction between safe and unsafe regions clearer, we plotted the solution to the Poisson equation $\nabla^2V=-6$ in 2D view in Fig. \ref{fig:2_6_2d}.
\begin{table}[ht]
    \centering
    \begin{tabular}{c|c|c||c|c|c|c}
        System & $\nabla^2V$ & $\sigma$ & $\mu_T$ & $\sigma_T$ & unsafe & no reach\\
        \hline
        Roomba & 0 & 0 & 85.600 & 54.603 & 12 & 0 \\
        \textbf{Roomba} & \textbf{-6} & \textbf{0} & \textbf{202.568} & \textbf{199.603} & \textbf{0} & \textbf{91}\\
        Roomba & 0 & 0.1 & 86.846 & 55.789 & 9 & 78 \\
        Roomba & -6 & 0.1 & 48.033 & 21.363 & 0 & 880\\
        \textbf{DiffDrive} & \textbf{0} & \textbf{0} & \textbf{258.502} & \textbf{168.726} & \textbf{0} & \textbf{55}\\
        DiffDrive & -6 & 0 & 96.077 & 28.211 & 0 & 730\\
        DiffDrive & 0 & 0.1 & 265.348 & 164.137 & 7 & 9\\
        DiffDrive & -6 & 0.1 & 209.452 & 171.338 & 0 & 118\\
        \textbf{CarRobot} & \textbf{0} & \textbf{0} & \textbf{380.014} & \textbf{230.751} & \textbf{30} & \textbf{124} \\
        CarRobot & -6 & 0 & 415.246 & 257.179 & 0 & 570 \\
        CarRobot & 0 & 0.1 & 264.435 & 205.930 & 45 & 942 \\
        CarRobot & -6 & 0.1 & 342.060 & 262.481 & 0 & 987 \\
    \end{tabular}
    \caption{Results for Problem II. Each system was run with 1,000 different randomly initialized initial conditions under harmonic and superharmonic CLBFs with deterministic and stochastic controllers.}
    \label{tab:CLBF2}
\end{table}
The numerical results are posted in Table \ref{tab:CLBF2}. Roomba achieves the best results under a superharmonic CLBF with a deterministic control policy with a $90.9\%$ safety rate. DiffDrive achieves the best results under a harmonic CLBF with a deterministic control policy with a $94.5\%$ safety rate. And CarRobot achieves the best results under a harmonic CLBF with a deterministic policy with a $84.6\%$ safety rate, although it has 30 trajectories converging to $\mathcal{S}_{unsafe}$ and 124 trajectories not converging to $\mathcal{S}_{goal}$. Furthermore, this table shows that deterministic control policies achieve better results than stochastic control policies and have significantly more trajectories that converge to $\mathcal{S}_{goal}$. 

Tables \ref{tab:CLBF1} and \ref{tab:CLBF2} show that deterministic policies under superharmonic CLBFs always avoid $\mathcal{S}_{unsafe}$ as expected since the Poisson equation makes the distinction between safe and unsafe regions more evident. But it comes with the drawback that there is a high risk that the trajectory will converge to a local minimum and not reach $\mathcal{S}_{goal}$. Deterministic policies generally outperform stochastic policies as the noise added to control inputs comes with the risk that it drives the trajectory away from the goal or towards an unsafe region. Despite using control inputs that solve the minimization problem in Equation~\eqref{eq:sup_inf}, in many episodes, the trajectory does not converge to $\mathcal{S}_{goal}$. This is likely due to numerical errors associated with estimating $\nabla V$ or with solving the system dynamics.

\subsection{2D Quadrotor}
In this problem, we control a quadrotor to land at a safe ground-level region while avoiding two obstacles. The dynamics of the quadrotor are in Appendix \ref{app:quad2d}. We are running the same problem as in \cite{du2023reinforcement} and the horizontally flipped problem in \cite{dawson2021safe}. The safe and unsafe subsets are given as the following.
\begin{align}
    S&=[-1,2]\times [0,2]\\
    \mathcal{S}_{goal}&=[-1,0.5]\times[0.25,0.75]\\
    \mathcal{S}_{unsafe}&=\partial S\cup C_7\cup C_8\cup C_9,
\end{align}
with the subdomains of the unsafe region given by
\begin{align}
    C_7&=[-1,0]\times[1.25,2]\\
    C_8&=[-1,2]\times [0,0.25]\\
    C_9&=[0.5,1]\times [0,1].
\end{align}
Initial conditions are sampled as:
\begin{align}
    x(0)&\sim U[1.4,1.6]\\
    z(0)&\sim U[0.3,1.5]\\
    x'(0),z'(0),\theta(0),\theta'(0)&\sim U[-0.05,0.05].
\end{align}
The numerical solution to the Harmonic CLBF for this problem is plotted in Fig. \ref{fig:quad2d}.

\begin{figure}[ht]
    \centering
    \includegraphics[width=0.8\linewidth,trim={0 0 0 5cm},clip]{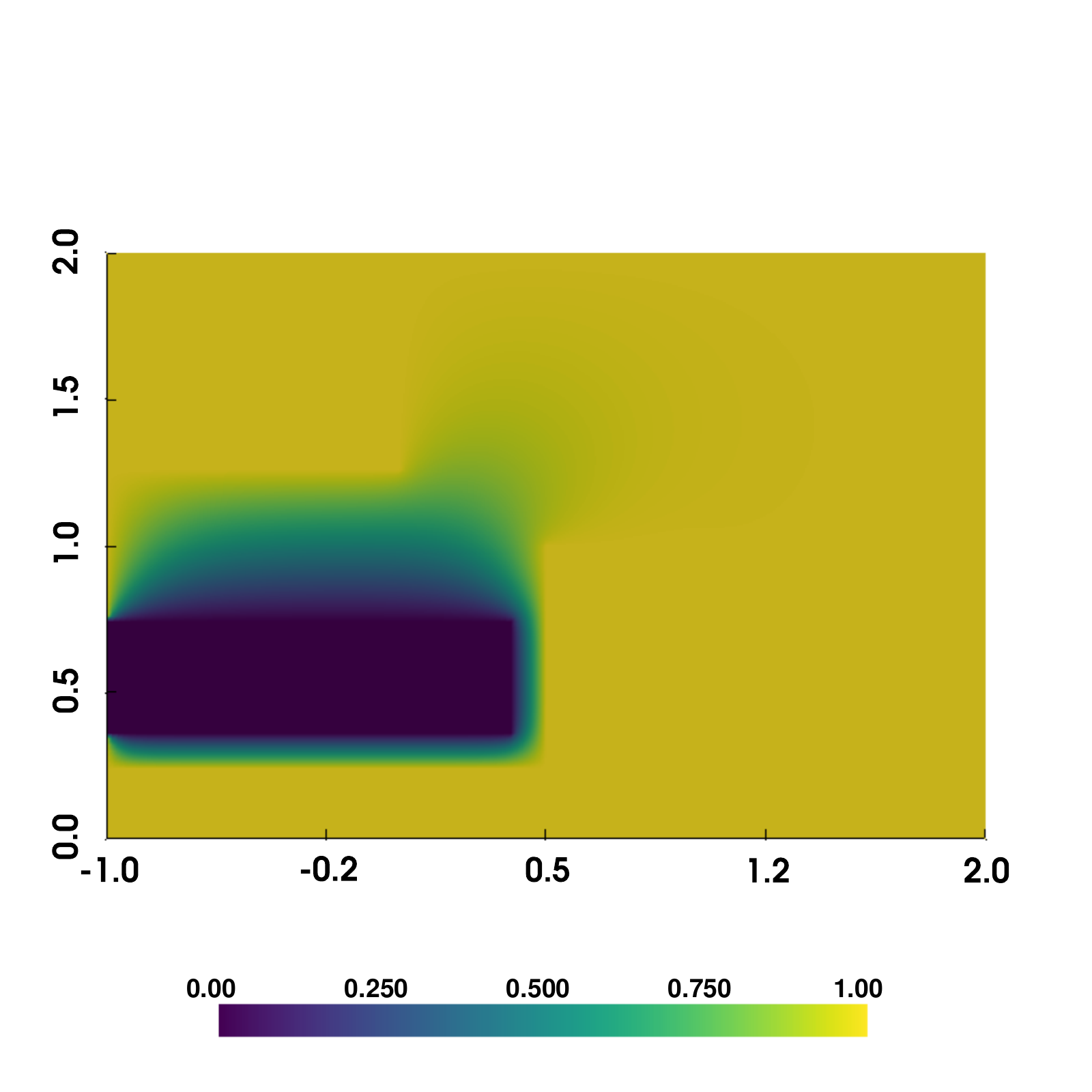}
    \caption{Countour plot of the harmonic CLBF for the 2D Quadrotor environment}
    \label{fig:quad2d}
\end{figure}

In this case, we use a PID controller $F(\x),M(\x)$, where $\x=(x,\dot x,z,\dot z,\theta,\dot\theta)$.
\begin{align}
    \phi_c&=\frac{1}{g}(\dot x-\tilde{V}_x(x,z))\\
    F&= m(g + \tilde{V}_z(x,z) - \dot z)\\
    M&= Ixx(-15\dot\theta + 18(\phi_c - z)),
\end{align}
where $\tilde{V}_x(x,y)=\frac{V_x(x,z)}{\sqrt{V_x(x,z)^2+V_z(x,z)^2}},\tilde{V}_z(x,z)=\frac{V_x(x,z)}{\sqrt{V_x(x,z)^2+V_z(x,z)^2}}$. This normalization is important because the scalar elements of $\nabla V$ take very small values near the initial condition.

We numerically solved the ode $\dot\x=f(\x, F(\x), M(\x))$ for a horizon of 2000 time steps using RK45 and repeated this for 100,000 different initial conditions. For $100\%$ of the initial conditions, the quadrotor converges to the goal region while avoiding the unsafe regions. It reaches the safe region at a mean time of $t=23.221$ seconds and a standard deviation of $1.765$.

The safety rate of the controller is higher than the safety rates for rCLBF-QP ($83\%$) and Robust-MPC ($53\%$) reported in \cite{dawson2021safe}. 

In Fig. \ref{fig:quad2d_traj}, we overlayed the trajectories of the quadrotor on 12 different initial conditions on a superharmonic CLBF plot with $\nabla^2V=-6$ that we computed only for presentation to make the distinction between safe and unsafe regions clearer. 
The trajectories in the plot show that the quadrotor can effectively avoid unsafe regions and enter the goal regions. 
%
\begin{figure}[ht]
    \centering
    \includegraphics[width=0.8\linewidth,trim={0 0 0 5cm},clip]{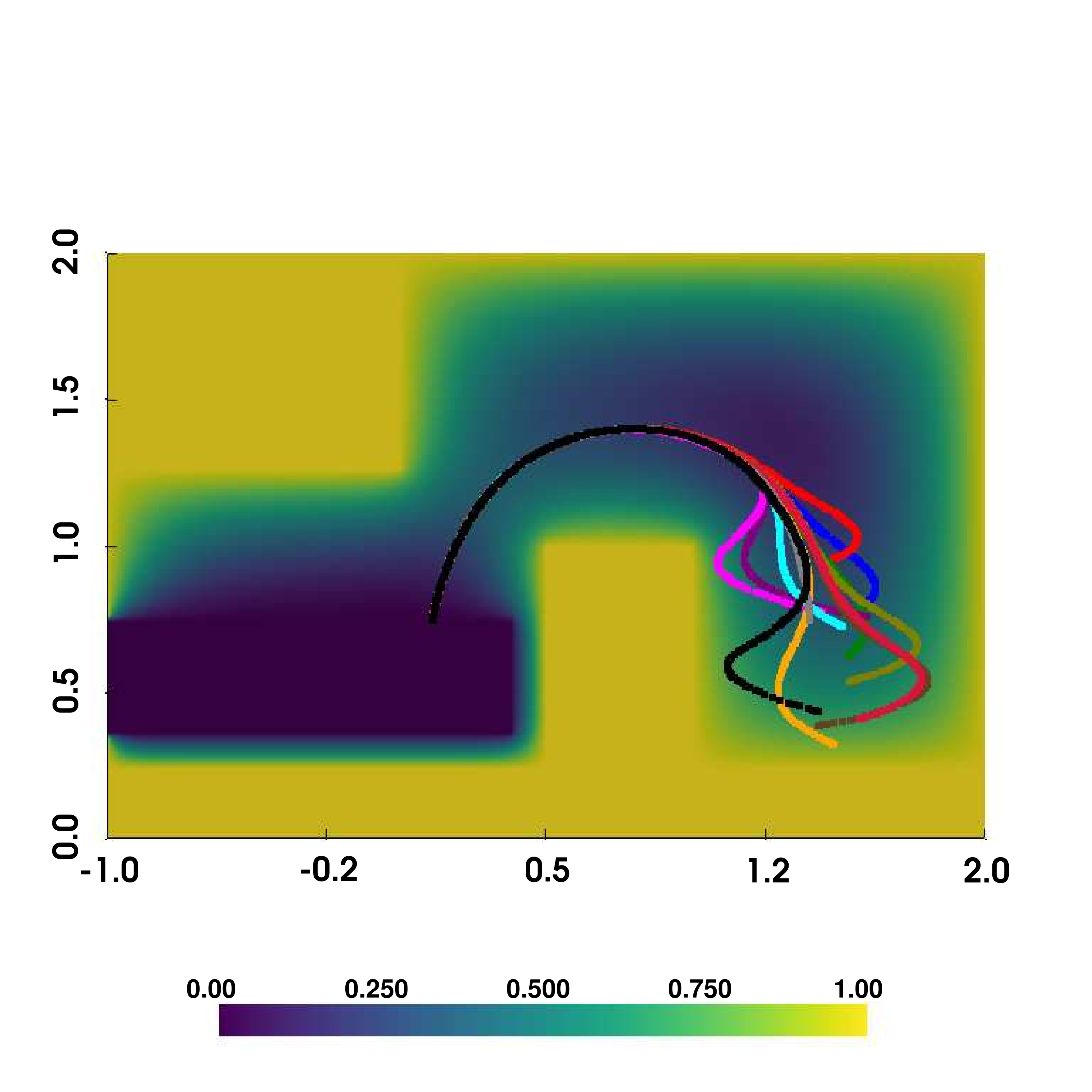}
    \caption{Plot of trajectories on the 2D Quadrotor environment with 12 different initial conditions.}
    \label{fig:quad2d_traj}
\end{figure}

\section{Conclusions}

In this paper, we introduced harmonic CLBFs that exploit the maximum principle that harmonic functions satisfy to encode the properties of CLBFs.
This paper is the first to unify harmonic functions with CLBFs for an application to control theory.
We select control inputs that maximize the inner product of the system dynamics with the steepest descent direction of the harmonic CLBF.
This has been applied to reach-avoid problems and demonstrated a low risk of entering unsafe regions while converging to the goal region. 




\section*{Appendix I: Dynamics of each system}

\subsection{Roomba}

The dynamics of the Roomba are as follows:
\begin{align} \nonumber
    \dot{x}=v\cos\theta, \quad \dot{y}=v\sin\theta, \quad \dot\theta=\omega,
\end{align}
where $v\in[-1,1]$ and $\omega\in[-1,1]$ are the control inputs.

The minimization problem to Equation \eqref{eq:sup_inf} is:
\begin{equation}\nonumber
    \inf_{v\in[-1,1],\omega\in[-1,1]}(v\cos\theta,v\sin\theta)(V_x(x,y),V_y(x,y))^T.
\end{equation}
The infimum is reached when $v$ maximizes the negative of the magnitude of the inner product, and $w$ is set to the direction that maximizes the magnitude of the inner product, which yields
\begin{align}
    v&=-\text{sign}(V_x(x,y)\cos\theta+V_y(x,y)\sin\theta), \nonumber\\
    w&=\text{sign}(-V_x(x,y)\cos\theta+V_y(x,y)\sin\theta). \nonumber
\end{align}
These control input pairs are used in computing the results in Tables \ref{tab:CLBF1} and \ref{tab:CLBF2}.

\subsection{DiffDrive}

The dynamics of the diff-drive robot are as follows:
\begin{align} \nonumber
    \dot x&=(u_L+u_R)\frac{r}{2}\cos\theta,
    \dot y=(u_L+u_R)\frac{r}{2}\sin\theta,\nonumber\\
    \dot\theta&=(u_R-u_L)\frac{r}{2d}, \nonumber
\end{align}
where $u_L\in[-1,1]$ and $u_R\in[-1,1]$ are the control inputs subject to the constraint $|u_L|+|u_R|\leq 1$. $r=0.1$ and $d=0.1$.

The minimization problem to Equation~\eqref{eq:sup_inf} is:
\begin{align} \nonumber
    \inf_{|u_L|+|u_R|\leq 1}&((u_L+u_R)\frac{r}{2}\cos\theta,(u_L+u_R)\frac{r}{2}\sin\theta)\nabla V(x,y).
\end{align}
The infimum is reached when $u_L+u_R$ maximizes the negative of the magnitude of the inner product, and $u_R-u_L$ is set to the direction that maximizes the magnitude of the inner product, which yields
\begin{align}
    u_R=[&\text{sign}(-V_x(x,y)\cos\theta+V_y(x,y)\sin\theta)\nonumber\\&-\text{sign}(V_x(x,y)\cos\theta+V_y(x,y)\sin\theta)]/2, \nonumber\\
    u_L=[&-\text{sign}(-V_x(x,y)\cos\theta+V_y(x,y)\sin\theta)\nonumber\\&-\text{sign}(V_x(x,y)\cos\theta+V_y(x,y)\sin\theta)]/2. \nonumber
\end{align}

\subsection{CarRobot}

The dynamics of the car-like robot are as follows:
\begin{align}\nonumber 
    \dot{x}=v\cos\theta, \quad
    \dot{y}=v\sin\theta, \quad
    \dot\theta=v\tan(\psi)/l, \quad
    \dot\psi=w,
\end{align}
where $v\in[-1,1]$ and $w\in[-1,1]$ are the control inputs subject to the constraint $|v|\leq |w|$. $l=0.1$. $\psi(0)$ is initialized to $0$. We set $l=0.1$.

The minimization problem to Equation \eqref{eq:sup_inf} is:
\begin{equation} \nonumber
    \inf_{|v|\leq|w|\leq 1}(v\cos\theta,v\sin\theta)(V_x(x,y),V_y(x,y))^T.
\end{equation}
The infimum is reached when $v$ maximizes the negative of the magnitude of the inner product, and $\theta$ is set to the direction that maximizes the magnitude of the inner product.
\begin{align} 
    v&=-\text{sign}(V_x(x,y)\cos\theta+V_y(x,y)\sin\theta), \nonumber\\
    w&=v\text{sign}(-V_x(x,y)\cos\theta+V_y(x,y)\sin\theta)-\text{sign}(\psi). \nonumber
\end{align}

\subsection{2D Quadrotor}
\label{app:quad2d}

We use the implementation of the 2D quadrotor environment by \cite{9849119}. The dynamics of the quadrotor are as follows:
\begin{align}
    \ddot{x}&=-F\sin\theta/m \nonumber\\
    \ddot{z}&=F\cos\theta/m-g \nonumber\\
    \ddot{\theta}&=M/I_{xx}, \nonumber
\end{align}
where $F,M\in [g/2-4,g/2+4]$ are the control inputs. $m=0.033$, $I_{xx}=2.31e-05$, $g=9.81$ and $d=0.046/\sqrt{2}$ are constants.






\section*{ACKNOWLEDGMENT}

The authors would like to thank Professor Yash Pant from the Department of Electrical and Computer Engineering, University of Waterloo for providing us with feedback for this paper.

\bibliography{root}
\bibliographystyle{plain}

\end{document}